\documentclass{article}
\usepackage[pdftex]{graphicx}
\usepackage{amsmath, amsthm, amssymb}

\newtheorem{definition}{Definition}
\newtheorem{theorem}{Theorem}
\newtheorem{example}{Example}
\newtheorem{lemma}{Lemma}
\newtheorem{corollary}{Corollary}
\title{New universal partizan rulesets and a new universal dicotic partisan ruleset}
\author{Koki Suetsugu}

\begin{document}

\maketitle
\begin{abstract}
Universal partizan ruleset is a ruleset in which every game value of partizan games can be appear as a position. So far, {\sc generalized konane} and {\sc turning tiles} have been proved to be universal partizan rulesets.
In this paper, we introduce two rulesets {\sc go on lattice} and {\sc beyond the door} and prove that they are universal partizan rulesets by using game tree preserving reduction. Further, we consider dicotic version of {\sc beyond the door}, and we prove the ruleset is a universal partisan dicotic ruleset.

\end{abstract}

\section{Introduction}
Combinatorial game theory (CGT) studies two-player perfect information games with no chance moves. We say a game is under {\em normal play} convention if the player who moves last is the winner and a game is {\em partizan} game if the options for both players can be different in some positions.
Here, we introduce some definitions and theorems of CGT for later discussion. For more details of CGT, see \cite{LIP, CGT}.

In this theory, the two players are called Left and Right.
Since the term ``game'' is polysemous, we refer to each position as a {\it game}. The description of what moves are allowed for a given position is called the {\it ruleset}.

A game is defined by Left and Right options recursively.
\begin{definition} \mbox{}
\begin{itemize}
\item $\{ \mid \}$ is a game, which is called $0$.
\item For games $G^L_1, G^L_2, \ldots, G^L_n, G^R_1, G^R_2, \ldots,$ and $ G^R_m,$ $G =  \{ G^L_1, G^L_2, \ldots, G^L_n \mid G^R_1, G^R_2, \ldots, G^R_m\}$ is also a game. $G^L_1, G^L_2, \ldots, G^L_n$ are called left options of $G$ and $G^R_1, G^R_2, \ldots, G^R_m$ are called right options of $G$.
\end{itemize}
Let $\mathbb{G}$ be the set of all games.
\end{definition}

In terms of the player who has a winning strategy, 
$\mathbb{G}$ is separated into four sets. Let $\mathcal{L}$, $\mathcal{R}$, $\mathcal{N}$, and $\mathcal{P}$ be the set of positions in which $\mathcal{L}$eft, $\mathcal{R}$ight, the $\mathcal{N}$ext player, and the $\mathcal{P}$revious player have winning strategies, respectively. 

The sets are called {\em outcomes} of the games. Every position belongs to exactly one of the four outcomes. For a game $G,$ let $o(G)$ be the outcome of $G$.
We define the partial order of outcomes as $\mathcal{L} > \mathcal{P} > \mathcal{R}, \mathcal{L} > \mathcal{N} > \mathcal{R}.$

The {\it disjunctive sum} of games is an important concept in Combinatorial Game Theory. For games $G$ and $H$, a position in which a player makes a move for one or the other on their turn is called a disjunctive sum of $G$ and $H$, or $G + H$. More precisely, it is as follows:

\begin{definition}
If the game trees of $G$ and $H$ are isomorphic, then we say these games are isomorphic or $G \cong H.$
\end{definition}

\begin{definition}
For games $G \cong \{G^L_1, G^L_2 \ldots G^L_n \mid G^R_1, G^R_2, \ldots, G^R_{m}\}$ and $H \cong \{ H^L_1, H^L_2 \ldots, H^L_{n'} \mid H^R_1, H^R_2, \ldots, H^R_{m'}\},
G+H \cong \{G + H^L_1, G+H^L_2, \ldots, G+H^L_{n'}, G^L_1 + H, G^L_2+H, \ldots, G^L_n + H \mid G+H^R_1, G+H^R_2, \ldots, G+H^R_{m'}, G^R_1+H, G^R_2+H, \ldots, G^R_m + H \}.$
\end{definition}

We also define equality, inequality and negative of games.
\begin{definition}
If for any $X, o(G + X)$ is the same as $o(H + X),$ then we say $G = H$. 
\end{definition}

\begin{definition}
If $o(G + H) \geq o(H + X)$ holds for any $X,$ then we say $G \geq H.$
On the other hand, if $o(G + H) \leq o(H + X)$ holds for any $X,$ then we say $G \leq H.$
We also say $G \gtrless H$ if $ G \not \geq H$ and $G \not \leq H.$
\end{definition}

\begin{definition}
For a game $G \cong  \{G^L_1, G^L_2, \ldots, G^L_n \mid G^R_1, G^R_2, \ldots, G^R_m\},$ let $-G \cong  \{ -G^R_1, -G^R_2, -G^R_m \mid -G^L_1, -G^L_2, \ldots, -G^L_n\}.$

$G + (-H)$ is denoted by $G - H$. 
\end{definition}

It is known that $(\mathbb{G}, +, =)$ is an abelian group and $(\mathbb{G}, \geq, =)$ is a partial order. 




The question arises here, will there be a ruleset in which for any game there is a position equal to the game? If the games appearing in each ruleset are restricted, then perhaps we should think in a narrower framework.
In fact, however, it is known that for every game, a position equal to the game appears in some rulesets.

\subsection{Universal partizan ruleset}
\begin{definition}
A ruleset is {\em universal partizan ruleset} if every value in $\mathbb{G}$ is equal to a position of the ruleset.
\end{definition}
Early results showed that {\sc generalized konane} and {\sc turning tiles} are universal partizan ruleset (\cite{CS, Sue}).
In this study, we will use the latter ruleset.
\begin{definition}
The ruleset of {\sc turning tiles} is as follows:
\begin{itemize}
    \item Square tiles are laid out. The front side is red or blue, and the back side is black. 
    \item Some pieces are on tiles.
    \item Each player (Left, whose color is bLue and Right, whose color is Red), in his/her turn, take a piece and move the piece straight on the tiles of his/her color.
    \item Tiles on which the piece pass over are turned over.
    \item The player who moves last is the winner.    
\end{itemize}
\end{definition}

\begin{center}
\begin{figure}[tb]
\includegraphics[height=6cm]{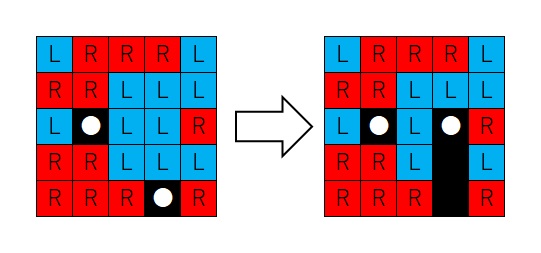}
\caption{Position and a move  in {\sc blue-red turning tiles}}
\label{fig_moveTT}
\end{figure}
\end{center}

Figure \ref{fig_moveTT} shows a position and a move by Left in {\sc blue-red turning tiles}.
Here, ``L'' means the tile is colored blue and ``R'' means the tile is colored red. The characters also correspond to the player who can move pieces on the tile. 

{\sc Turning tiles} is proved to be universal partizan ruleset even if the number of pieces is restricted to be only one.

To distinguish this ruleset from the ruleset defined below, we will also refer to it as {\sc blue-red turning tiles}.

For games that use two colors, red and blue, corresponding to two players, we often consider a variant that adds green, which can be used by both players.
For example, in {\sc blue-red-green hackenbush}, Left can remove blue or green edges and Right can remove red or green edges.  
From this point of view, we consider a varant of {\sc turning tiles}.

\begin{definition}
The ruleset of {\sc blue-red-green turning tiles} is as follows:
\begin{itemize}
    \item Square tiles are laid out. The front side is red, blue, {\em or green}, and the back side is black. 
    \item Some pieces are on tiles.
    \item Each player (Left, whose color is bLue and Right, whose color is Red), in his/her turn, take a piece and move the piece straight on the tiles of his/her color {\em or of  green}.
    \item Tiles on which the piece pass over are turned over.
    \item The player who moves last is the winner.    
\end{itemize}
\end{definition}

\begin{figure}[htb]
    \centering    \includegraphics[keepaspectratio, width = 12cm]{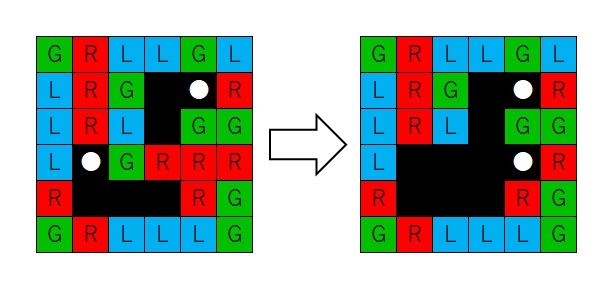}
\caption{Position and a move in {\sc blue-red-green turning tiles}}
    \label{fig:positions}
\end{figure}

Figure \ref{fig:positions} is an example of a position and a move by Right in {\sc blue-red-green turning tiles}.

Obviously, {\sc blue-red-green turning tiles} is also a universal partizan ruleset because every position in {\sc blue-red turning tiles} can be appear in {\sc blue-red-green turning tiles}.

As we have seen here, if two rulesets have an inclusion relation in terms of the sets of  positions, it can be used for proving universality of the rulesets.

\begin{theorem}
Let $\Gamma$ and $\Delta$ be rulesets and assume that $\Gamma$ be a universal partizan ruleset.
If for every position $g \in \Gamma$, there is at least one position in $\Delta$ whose game value is the same as $g$, then $\Delta$ is also a universal partizan ruleset.
\end{theorem}

\begin{proof}
This is trivial from the definition of universal partizan ruleset.
\end{proof}

\begin{corollary}
\label{col1}
Let $\Gamma$ and $\Delta$ be rulesets and assume that $\Gamma$ be a universal partizan ruleset.
If for every position $g \in \Gamma$, there is at least one position in $\Delta$ whose game tree is the same as $g$, then $\Delta$ is also a universal partizan ruleset.
\end{corollary}

If a ruleset is proved to be universal partizan ruleset by using Corollary \ref{col1}, we say that it is proved by game tree preserving reduction.

\subsection{Dicotic positions and dicotic rulesets}

We also consider {\em dicotic} (or, {\em all-small}) rulesets.  
A dicotic position is a position such that Left has a dicotic option if and only if Right has a dicotic option. That is, during the play, each player has at least one option except for the terminal position.
For example, $0 = \{ \mid \}, * = \{ 0 \mid 0\},\uparrow = \{ 0 \mid *\} = \{ 0 \mid \{ 0 \mid 0\} \},$ and $\downarrow = -\uparrow$ are dicotic positions but $1 = \{0 \mid \}$ is not a dicotic position because Right has no option while Left has an option $0.$

Let $n \cdot G = \overbrace{G + \cdots + G}^{n\text{ times}}$.

\begin{definition}
A position $G$ is an {\em infinitesimal} if for any positive integer $n$, $n \cdot G < 1$ holds.
\end{definition}

The reason why dicotic positions are sometimes called all-small positions is following well-known result:
\begin{theorem}
Every dicotic position $G$ is an infinitesimal.
\end{theorem}

Among dicotic positions, we remark some results on $\uparrow$ and $*$.
\begin{theorem}
\label{thm:nup}
For any positive integer $n,$ the following equations holds.
\begin{eqnarray}
        n \cdot \uparrow &=& \{ 0 \mid (n-1) \cdot \uparrow +* \}, \nonumber \\
        n \cdot \uparrow+* &=& 
        \left \{
        \begin{array}{cc}
             \{ 0 \mid (n-1) \cdot \uparrow\}& (\text{If }n>1.)  \\
             \{0,*\mid 0\}&(\text{If }n=1.)
        \end{array} \right . \nonumber \\
        \uparrow &\mid \mid& *, \nonumber \\
        n \cdot \uparrow & > & *(\text{If }n > 1.) \nonumber 
    \end{eqnarray}
    
\end{theorem}

\begin{theorem}
\label{thmuplim}
Let $G$ be an infinitesimal born by day $n+1,$ that is, the height of $G$'s game tree is less than or equal to $n+1$. Then, $G \leq n \cdot \uparrow$ or $G \leq n \cdot \uparrow + *$.
\end{theorem}
For proofs and details on these theorems, see  CGT textbooks like \cite{LIP, CGT}.
Since $2 \cdot \uparrow > *,$ The statement of Theorem \ref{thmuplim} can be rephrased as $G < (n + 2) \cdot \uparrow$ {\em and} $G < (n + 2) \cdot \uparrow + *$. Thus, the following corollary holds:   
\begin{corollary}
\label{cor:up}
    For any dicotic position $G$, there is an integer $N$ and for any $n \geq N,~G < n \cdot \uparrow, G < n \cdot \uparrow + *, G > n\cdot \downarrow,$ and $G > n\cdot \downarrow +*$ holds.
\end{corollary}

We will use this corollary in Section \ref{secDR}.

A dicotic ruleset is a ruleset in which every position is a dicotic position and a universal partisan dicotic ruleset is a dicotic ruleset in which every partisan dicotic value can appear as a position. In \cite{LN23}, to find a universal partisan dicotic ruleset is introdiced as an open problem. In this paper, we solve it.

In the next section, we introduce two rulesets {\sc go on lattice} and {\sc beyond the door}, and prove that they are universal partizan ruleset by game tree preserving reduction. 
In Section~\ref{secDR}, we consider dicotic version of {\sc beyond the door} and prove that the ruleset is a universal partisan dicotic ruleset.
The final section presents the conclusions.


\section{New universal partizan rulesets}
\subsection{{\sc Go on lattice}}
\begin{definition}
The rule of {\sc go on lattice} is as follows:
    \begin{itemize}
        \item There is a lattice graph. There are pieces on some nodes. The edges are colored red, blue, or dotted.
        \item A player, in his/her turn, chooses a piece and moves it straight on edges colored his/her color.
        \item After a piece passed a node, every piece cannot get on or pass the node.
        \item If a player moves a piece to a node adjacent to a dotted edge, then the edge changes to solid edge colored by the opponent's color.
        \item The player who moves last is the winner.
    \end{itemize}
\end{definition}

\begin{figure}[htb]
    \centering
    \includegraphics[width = 12cm]{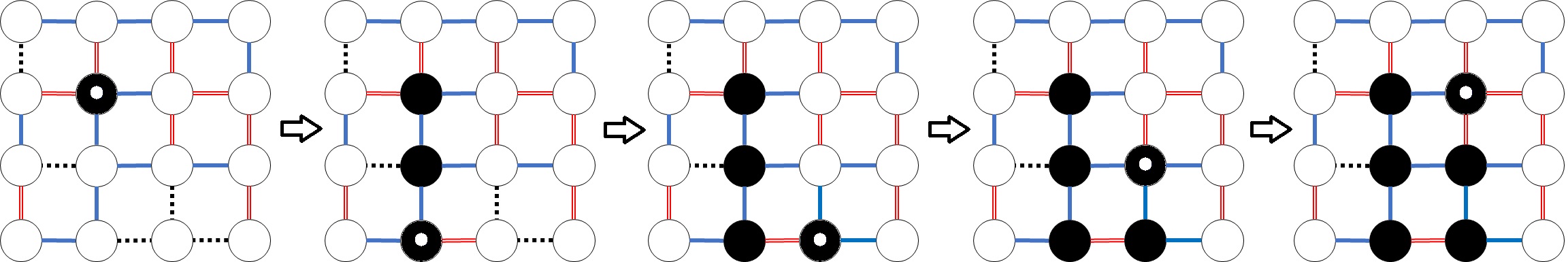}
    \caption{Play of {\sc go on lattice}}
    \label{fig:posgol}
\end{figure}
Figure \ref{fig:posgol} is a play of {\sc go on lattice}. We use double line for red edges for monochrome printing.
\begin{theorem}
    {\sc Go on lattice} is a universal partizan ruleset.
\end{theorem}
\begin{proof}
Let $f$ be a function from a position in {\sc turning tiles} to a position in {\sc go on lattice} as follows:

Let $G$ be a position in {\sc turning tiles}. In  $f(G)$ there are as many nodes as tiles in $G$. The tiles in $G$ and the nodes in $f(G)$ are arranged exactly the same.
For each piece on a tile in $G$, there is a corresponding piece on the node corresponds to the tile.
For any adjacent tiles $A$ and $B$ in $G$, let $A'$ and $B'$ are corresponding nodes in $f(G)$.
If the color of $A$, and $B$ are the same, then edge between $A'$ and $B'$ is solid and the same color as $A$ and $B$. 
If there is a piece on $A$ or $B$, then the edge between $A'$ and $B'$ is solid and the color is the same as the other tile.
Finally, if the color of $A$ and $B$ are different and no piece is on each tile, then the edge between $A'$ and $B'$ is dotted line.

\begin{figure}[htb]
    \centering
    \includegraphics[width = 10cm]{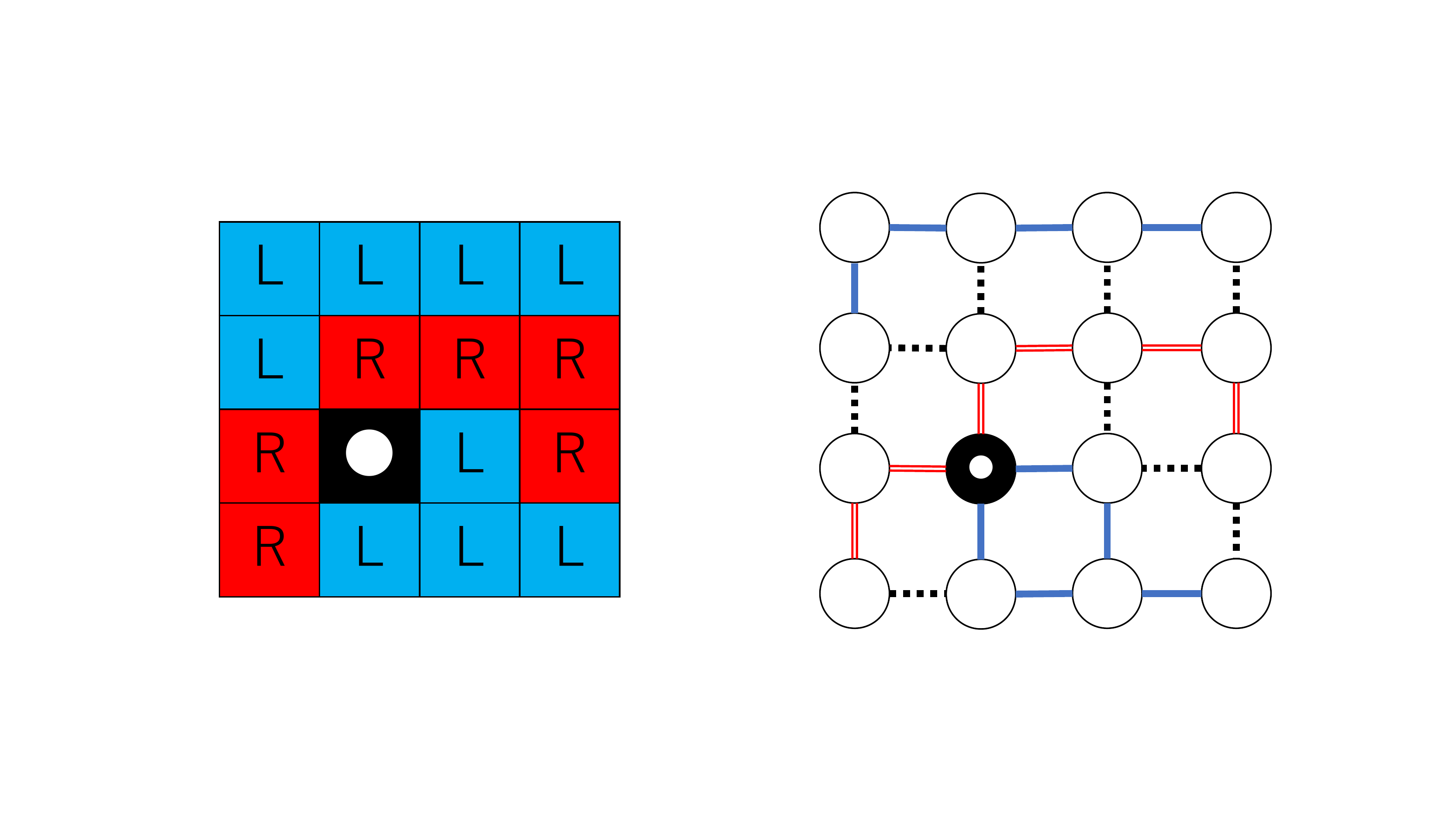}
    \caption{Corresponding positions in {\sc blue-red turning tiles} and {\sc go on lattice}}.
    \label{fig:cor_gol}
\end{figure}

Figure \ref{fig:cor_gol} shows this corresponding.
Here, the game tree of $G$ and $f(G)$ are isomorphic.
We prove that every move in one game has a corresponding move in the other game.
Assume that in $G$ Left can move a piece on tile $A_0$ to tile $A_n$ through tiles $A_1, A_2, \ldots, A_{n-1}$. Then, $A_1, A_2, \ldots, A_n$ are blue tiles. Let $A_0', A_1',  \ldots, A_n'$ be the corresponding nodes in $f(G)$. 
Let $(A', B')$ be the edge between $A'$ and $B'$.
Then, from the definition of $f$, all of $(A'_1, A'_2), (A'_2, A'_3), \ldots, (A'_{n-1}, A'_n)$ are blue edge.
In addition, if $A_0$ was a blue tile before turning, then $(A'_0, A'_1)$ is a blue edge, and if $A_0$ was a red tile, $(A'_0, A'_1)$ had been a dot edge and after Right moved the piece to $A'_0,$ it changed to a blue edge.
Therefore for both case, $(A'_0, A'_1)$ is a blue edge and Left can move a piece
from $A'_0$  to $A'_n$ through $A'_1, A'_2, \ldots, A'_{n-1}$.
Conversely, assume that in $f(G)$, Left can move a piece from $A'_0$ to $A'_n$ through $A'_1, A'_2, \ldots, A'_{n-1}$.
Then, in $G$, all corresponding tiles $A_1, A_2, \ldots, A_n$ are blue tiles. Therefore, Left can move a piece from $A_0$ to $A_n$ through $A_1, A_2, \ldots, A_{n-1}$ in the corresponding position in {\sc turning tiles}.
Similar proof holds for Right's moves.

Thus, from Corollary \ref{col1}, {\sc go on lattice} is a universal partizan ruleset.
\end{proof}
\subsection{{\sc Beyond the door}}
\begin{definition}
The rule of {\sc beyond the door} is as follows: 
\begin{itemize}
\item 
Square rooms are arranged in a grid pattern. There are doors between the rooms. The front and back of the doors are painted red, blue, or black. There are pieces in several rooms.
\item 
A player, in his/her turn, chooses a piece and moves it in a straight line. When a piece moves beyond the door, the color of the piece's side of the door must be the player's color.
\item After a piece passed a room, every piece can not enter the room.
\item The player who moves last is the winner.
\end{itemize}
\end{definition}

For monochrome printing, we use directed graph with two types of arrowheads. See Figure \ref{fig:pos_door}.
Each node corresponds to each room and each edge corresponds to each door.
If Left (bLue) can move a piece from a room to an adjacent room, we use filled arrowhead in the corresponding graph. Similarly, if Right (Red) can move a piece from a room to an adjacent room, we use open arrowhead in the corresponding graph. We omit the corresponding edge if both side of the door is colored black.

\begin{figure}
    \centering
    \includegraphics[width = 12cm]{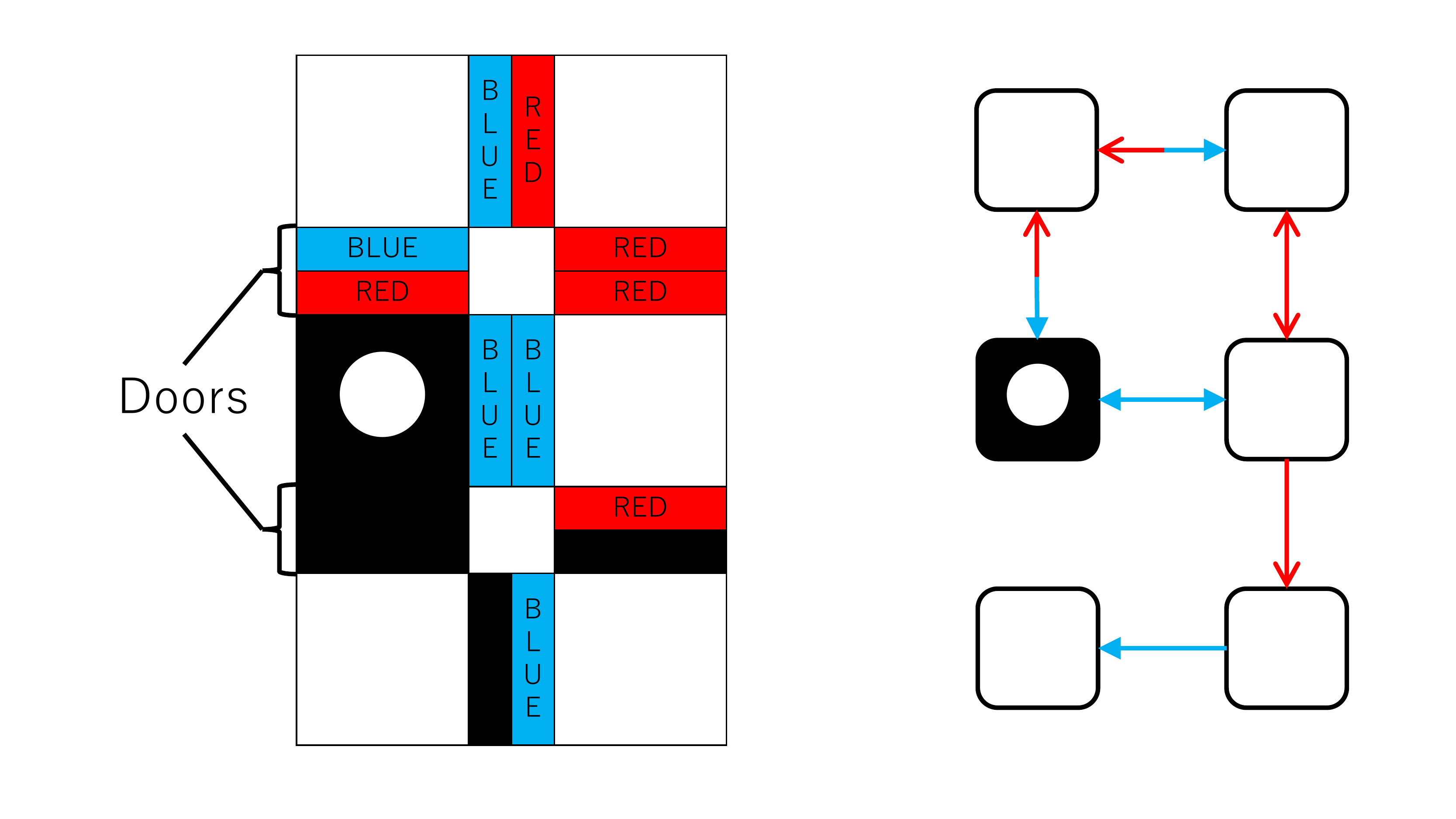}
    \caption{Position in {\sc beyond the door} and its representation by directed graph with two types of arrowheads}
    \label{fig:pos_door}
\end{figure}

\begin{theorem}
    {\sc Beyond the door} is a universal partizan ruleset.
\end{theorem}
\begin{proof}
Let $f'$ be a function from a position in {\sc turning tiles} to a position in {\sc beyond the door} as follows:

Let $G$ be a position in {\sc turning tiles}. In  $f'(G)$ there are as many rooms as tiles in $G$ and the tiles in $G$ and the rooms in $f'(G)$ are arranged exactly the same.
For each piece on a tile in $G$, there is a corresponding piece in the room corresponding to the tile.
For any adjacent tiles $A$ and $B$ in $G$, let $A'$ and $B'$ are corresponding rooms in $f'(G)$.
The color of the door separating $A'$ and $B'$ matches the color of $A$ when viewed from $B'$ and corresponds to the color of $B$ when observed from $A'$.

\begin{figure}[htb]
    \centering
    \includegraphics[width = 12cm]{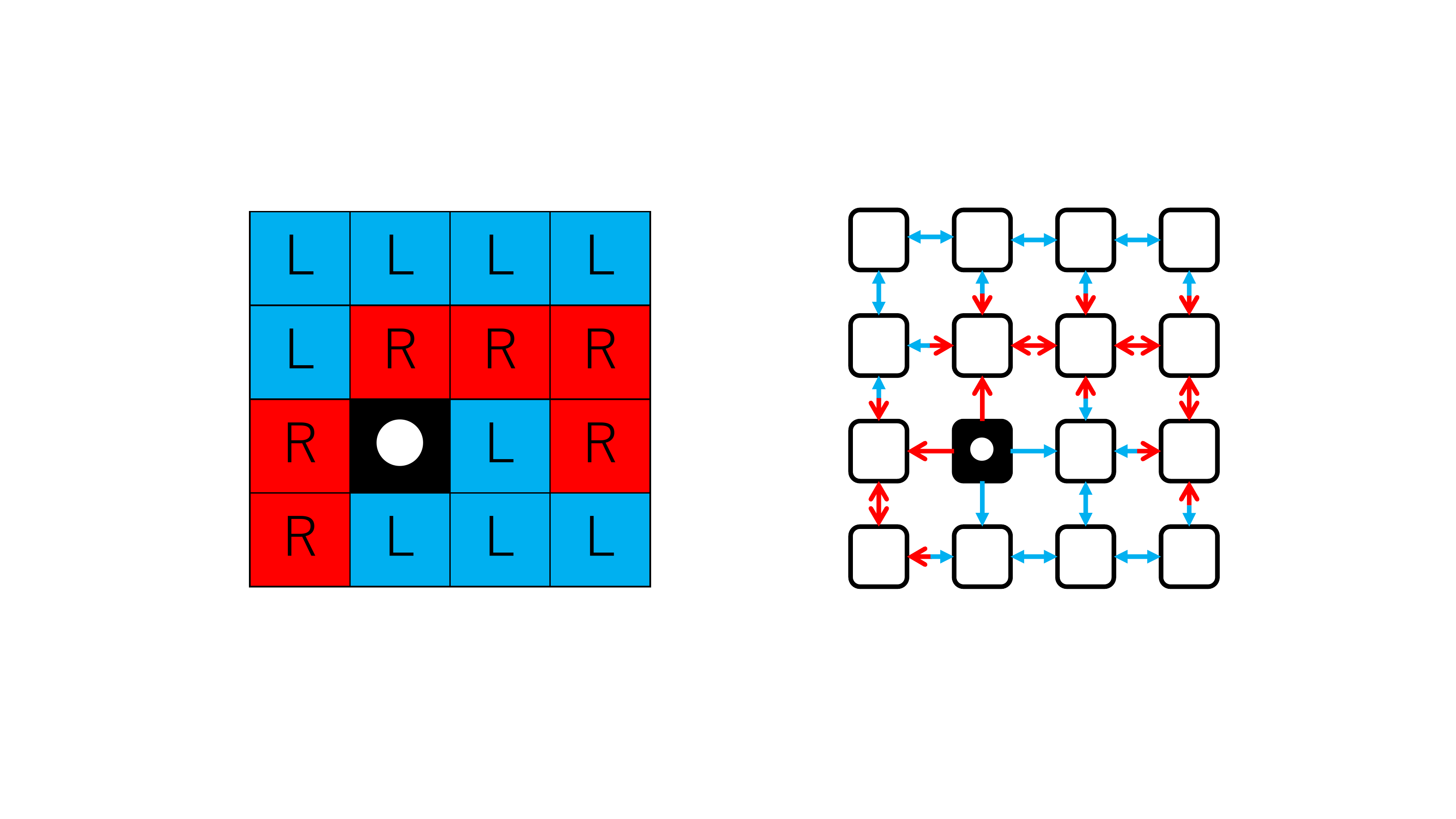}
    \caption{Corresponding positions in {\sc blue-red turning tiles} and {\sc beyond the door}}.
    \label{fig:cor_door}
\end{figure}

Figure \ref{fig:cor_door} shows this corresponding.
Here, the game tree of $G$ and $f'(G)$ are isomorphic.
We prove that every move in one game has a corresponding move in the other game.
Assume that in $G$ Left can move a piece on tile $A_0$ to tile $A_n$ through tiles $A_1, A_2, \ldots, A_{n-1}$. Then, $A_1, A_2, \ldots, A_n$ are blue tiles. Let $A_0', A_1',  \ldots, A_n'$ be the corresponding rooms in $f'(G)$. 
Let $A' \rightarrow B'$ be the color of the door between $A'$ and $B'$ on the $A'$ side.
Then, from the definition of $f'$, all of $A'_0 \rightarrow A'_1, A'_1 \rightarrow A'_2, \ldots, A'_{n-1} \rightarrow A'_n$ are blue.
Therefore, Left can move a piece from $A'_0$  to $A'_n$ through $A'_1, A'_2, \ldots, A'_{n-1}$.
Conversely, assume that in $f(G)$, Left can move a piece from $A'_0$  to $A'_n$ through $A'_1, A'_2, \ldots, A'_{n-1}$.
Then, in $G$, all corresponding tiles $A_1, A_2, \ldots, A_n$ are blue tiles. Therefore, Left can move a piece from $A_0$ to $A_n$ through $A'_1, A'_2, \ldots, A'_{n-1}$ in the corresponding position in {\sc turning tiles}.
Similar proof holds for Right's moves.

    Thus, from Corollary \ref{col1}, {\sc beyond the door} is a universal partizan ruleset.
\end{proof}

If $f$ has an inverse function, {\sc blue-red turning tiles} and {\sc go on lattice} has one-to-one correspondence between the positions of each ruleset. Similar thing can be said for $f'$ and {\sc beyond the door}.
However, $f$ and $f'$ have no inverse functions.
For instance, Fig. \ref{fig:counter} shows positions in {\sc go on lattice} and {\sc beyond the door}. No position in {\sc blue-red turning tiles} is mapped to these positions by $f$ and $f'$  because depending on the order of moves, both Left and Right may move pieces to the same node or the same room in these positions while in {\sc blue-red turning tiles} it never happen that both player can move to a piece on the same tile.

Thus, in a sense, {\sc go on lattice} and {\sc beyond the door} are more complex than {\sc blue-red turning tiles}, despite all of them are belonging to the class of universal partisan rulesets.

Note that the computational complexity of the problem determining the winner from given position is PSPACE-complete for each ruleset. It is proved in \cite{Yos2023}.

\begin{figure}
    \centering
    \includegraphics[width = 12cm]{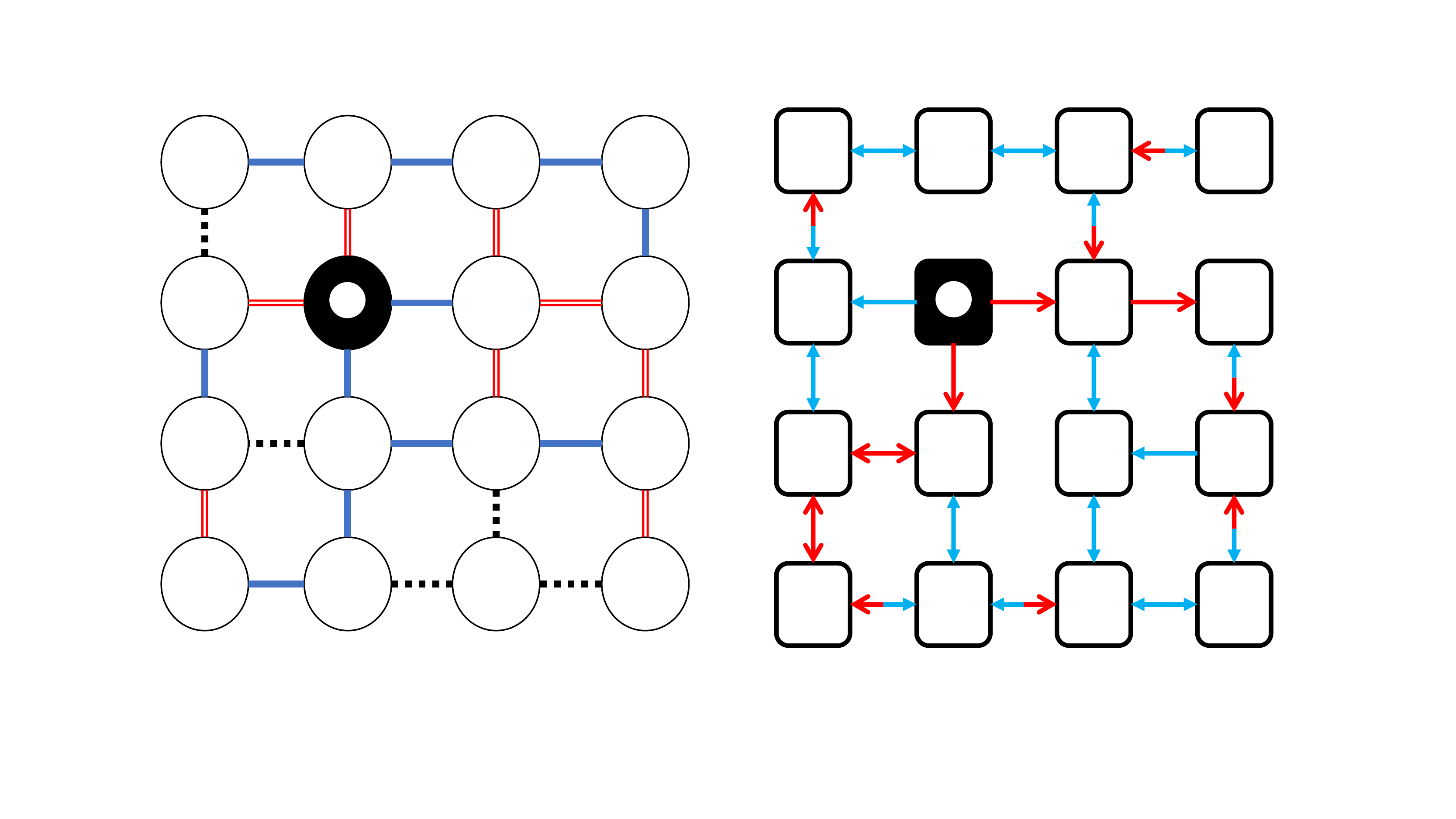}
    \caption{$f$ and $f'$ have no inverse functions.}
    \label{fig:counter}
\end{figure}

\section{Universal partisan dicotic ruleset}
\label{secDR}

In this section, we consider universal partisan dicotic ruleset.
Note that the rulesets discussed above are universal partisan rulesets but not universal partsan dicotic rulesets because in the rulesets positons are not always dicotic positions.

First, we consider {\em dicotic function} $\delta$.
\begin{definition}
Let $G$ be a position. Then $\delta(G)$ is a position such that every empty set in $G$ is replaced with $0$.
\end{definition}

\begin{example}
$$\delta(\{ \mid \}) = \{0 \mid 0\} = *,$$
$$\delta(\{ 0 \mid \}) = \{* \mid 0\} = \downarrow,$$
$$\delta(\{ 0 \mid 0\}) = \{* \mid *\} = 0.$$
\end{example}

Note that even if $A = B,$ sometimes $\delta(A) \neq \delta(B).$ For example, 
consider a position $2 = \{1 \mid \} = \{1 , 0 \mid \}.$ Here, $$\delta(\{1 \mid \}) = \{ \{* \mid 0\} \mid 0\} = \{ \downarrow \mid 0\}$$ and $$\delta(\{1, 0 \mid \}) = \{ \{* \mid 0\}, * \mid 0\} = \{ \downarrow, * \mid 0\}.$$ 
Here, $\{ \downarrow \mid 0\} \neq \{ \downarrow, * \mid 0\}$ because $o(\{ \downarrow \mid 0\} - \{ \downarrow, * \mid 0\} )= \mathcal{R}.$

Next, we introduce dicotic version of above rulesets.
\begin{definition}
The ruleset $\delta$-{\sc turning tiles} is the same as {\sc (blue-red) turning tiles}, except for the following added move:
\begin{itemize}
\item A player can remove a piece if and only if the player cannot move the piece anywhere.
\end{itemize}

Rulesets $\delta$-{\sc go on lattice} and $\delta$-{\sc beyond the door} are defined in a similar way.
\end{definition}

Table \ref{fig:exampledelta} shows how the value changes for positions in {\sc  turning tiles} and $\delta$-{\sc turning tiles}.

\begin{table}[htb]
\begin{tabular}{c}
    \centering
    \includegraphics[width = 12cm]{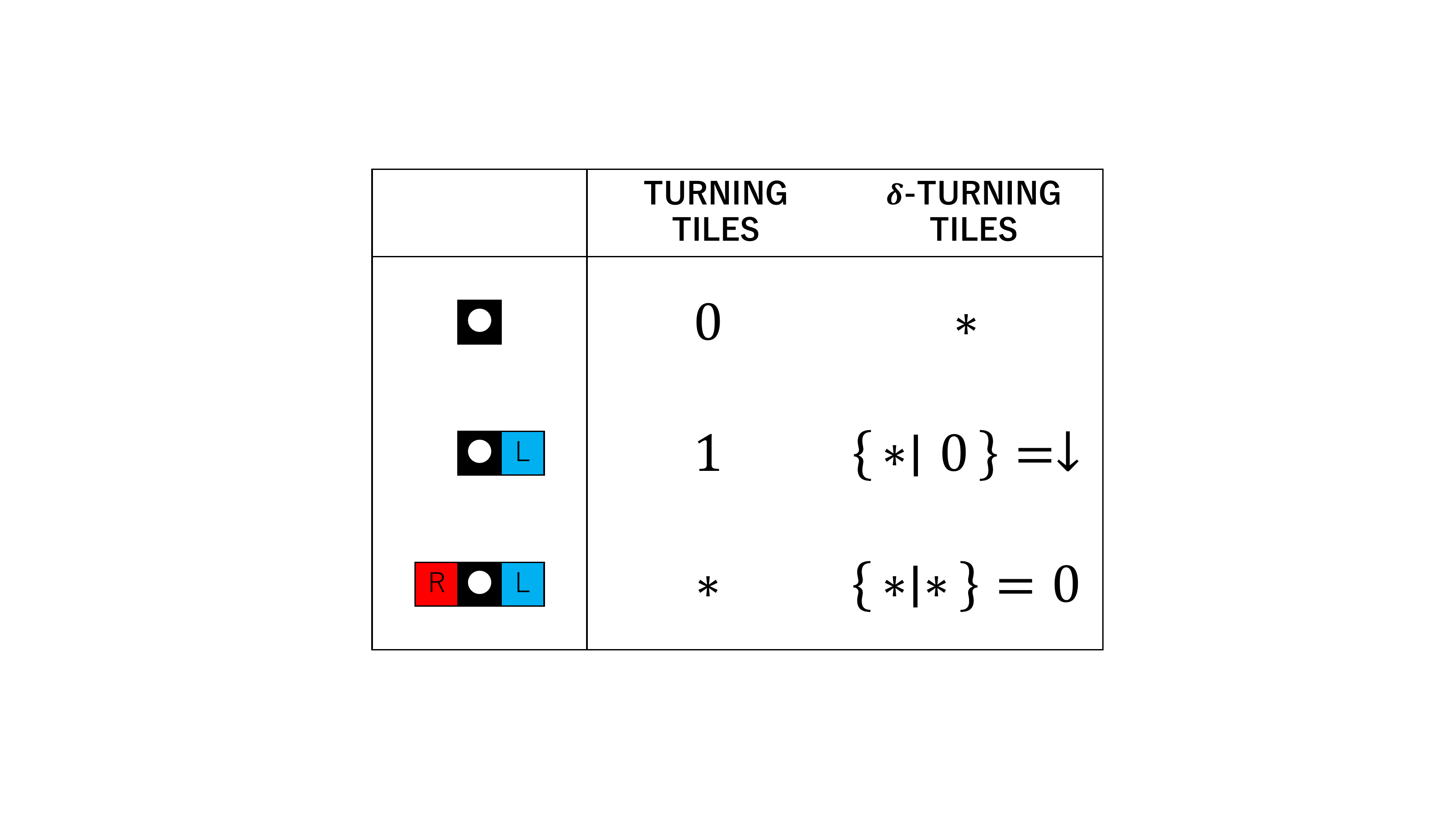}

    \end{tabular}  
    \caption{Positions and values in {\sc turning tiles} and $\delta$-{\sc turning tiles}}
    \label{fig:exampledelta}
\end{table}

Note that we cannot immediately say that these dicotic version of rulesets are universal partisan dicotic rulesets even though the original rulesets are universal partisan rulesets.
Consider a dicotic position $D$. Let $D'$ be a position such that every $0$ in $D$ is replaced with emptysets. Then in a universal partisan ruleset, say in {\sc blue-red turning tiles}, there is a position $X$ such that $X = D'$. However, as mentioned above, even if $X = D'$, it may be $\delta(X) \neq \delta(D')$, so we cannot say $D$ is in $\delta$-{\sc turning tiles.}

On the other hand, by following discussion, $\delta$-{\sc beyond the door} is proved to be a universal partisan dicotic ruleset.

\begin{lemma}
An $(n+1)$-room path where only Left can move, with direction changing at each step  has value $n\cdot \downarrow$ or $n \cdot \downarrow + *$ depends on $n$ is odd or $n$ is even, respectively.
\end{lemma}

\begin{proof}
We prove this theorem by induction.
If $n = 0$, then there is only one move and both player can only remove the piece, so the value is $\{ 0 \mid 0\} = * = 0\cdot \downarrow + *$.

Assume that $n$ is odd and the statement holds for $n-1$. Then, consider $(n + 1)$ room path. Since the direction changes every step, Left has only one option, which is $n$ room path where the direction changes with every step. From the induction hypothesis, the value is $(n - 1) \cdot \downarrow + *$. Right cannot move the piece anywhere and can only remove the piece.
Therefore, the value of this $(n + 1)$-path is $\{ (n-1) \cdot \downarrow + *\mid 0\} = n \cdot \downarrow$ from Theorem \ref{thm:nup} and $\downarrow = -\uparrow$.

Similarly, if $n$ is even, then the value is $\{(n - 1)\cdot \downarrow \mid 0\} = n \cdot \downarrow + *$.
\end{proof}

Figure \ref{fig:pathdown} shows such positions.
\begin{figure}[htb]
    \centering
    \includegraphics[width = 12cm]{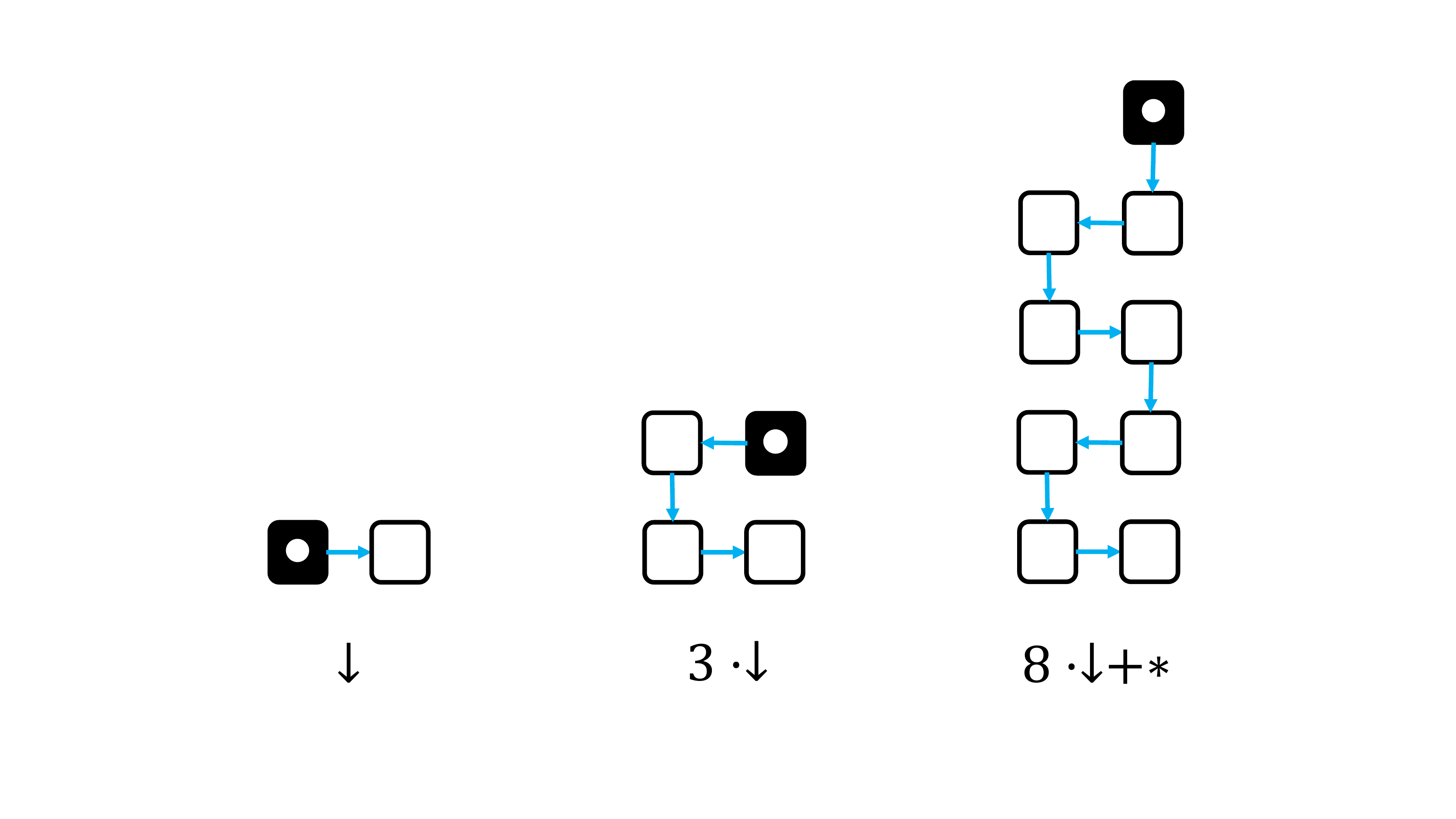}
    \caption{$(n+1)$-room path where only Left can move, with direction changing at each step.}
    \label{fig:pathdown}
\end{figure}

\begin{theorem}
$\delta$-{\sc beyond the door} is a universal partisan dicotic ruleset.
\end{theorem}
\begin{proof}
Note that from Corollary \ref{cor:up}, for any dicotic potision $G$, if $n$ is sufficiently large, then $n \cdot \downarrow < G$ and $n \cdot \downarrow + * < G$. 

We prove this theorem by induction.
We use similar way to the proof of Theorem 1 of \cite{Sue}.

Assume that each dicotic option $G^L_1,  \ldots, G^L_n, G^R_1,  \ldots, G^R_m$ has already constructed in the same way.
We show a dicotic position $G \cong \{ G^L_1, \ldots, G^L_n \mid G^R_1, \ldots, G^R_m  \}$ can be constructed as Figures \ref{fig:dprove1}, \ref{fig:dprove2} and \ref{fig:dprove3}.
The entire position is very large, so we split the diagrams to primarily show how $G^L_1$ is connected. The other Left options are connected in a similar way and the part of Right options are constructed in a symmetric way. Note that the options are rotated and reflected.
The first, second, and third connection points play the same role as they do in Theorem 1 of \cite{Sue}. 
We also assume that the parts represented by dots are sufficiently long. 

\begin{figure}[htb]
    \centering
    \includegraphics[width = 12cm]{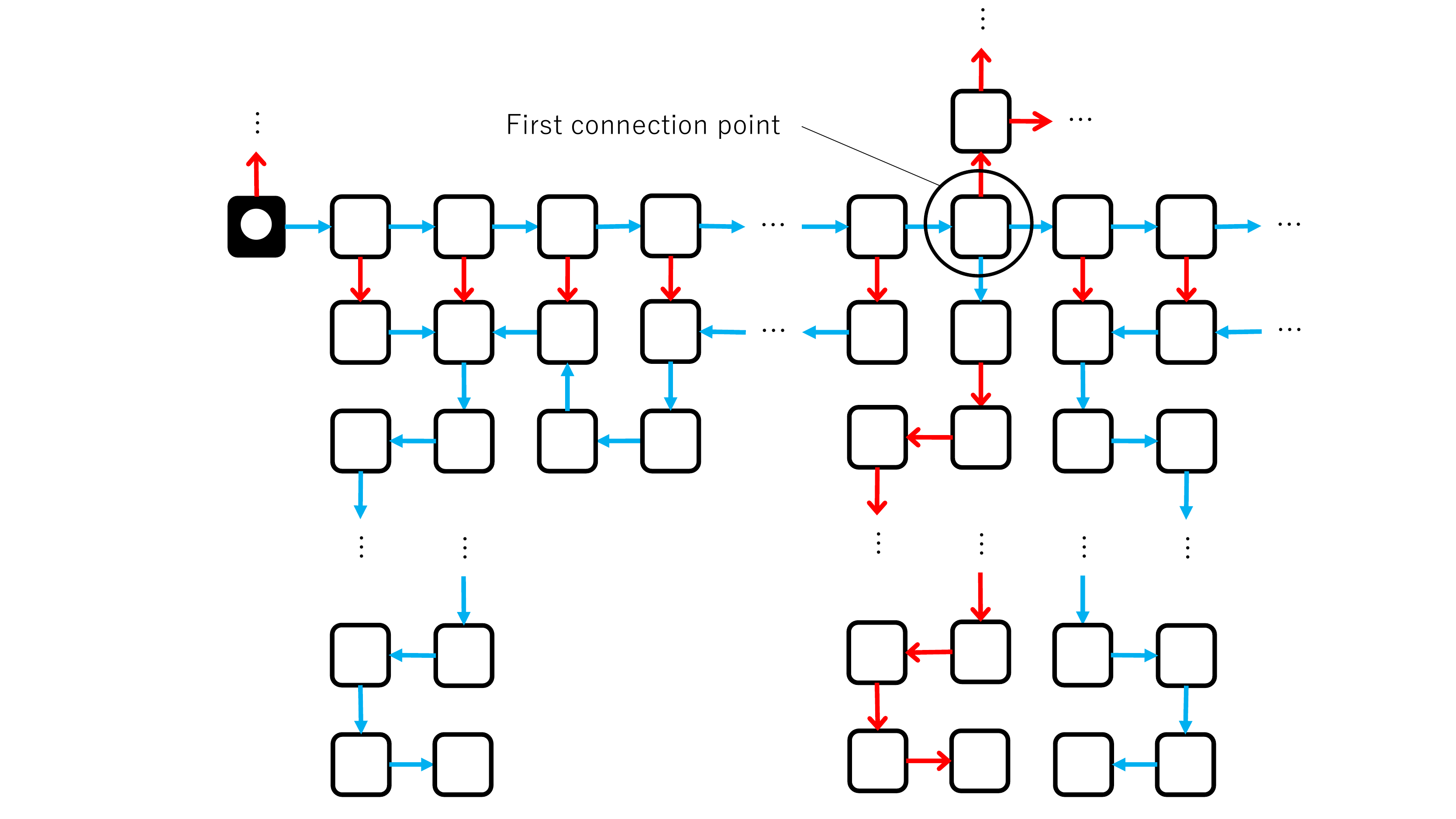}
    \caption{The section between the piece and first connection point of $G^L_1$}
    \label{fig:dprove1}
\end{figure}

\begin{figure}[htb]
    \centering
    \includegraphics[width = 12cm]{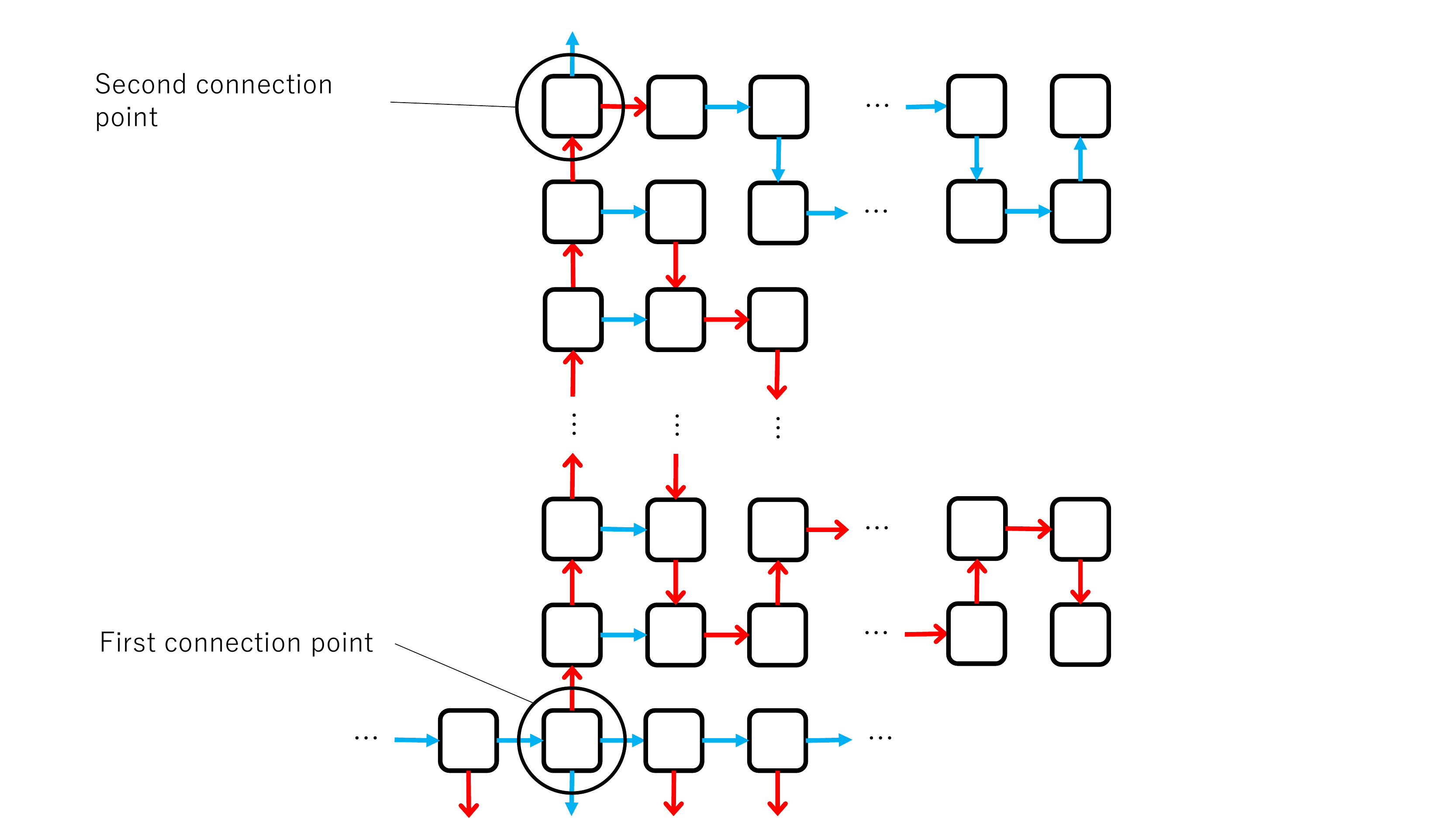}
    \caption{The section between first connection point and second connection point}
    \label{fig:dprove2}
\end{figure}

\begin{figure}[htb]
    \centering
    \includegraphics[width = 12cm]{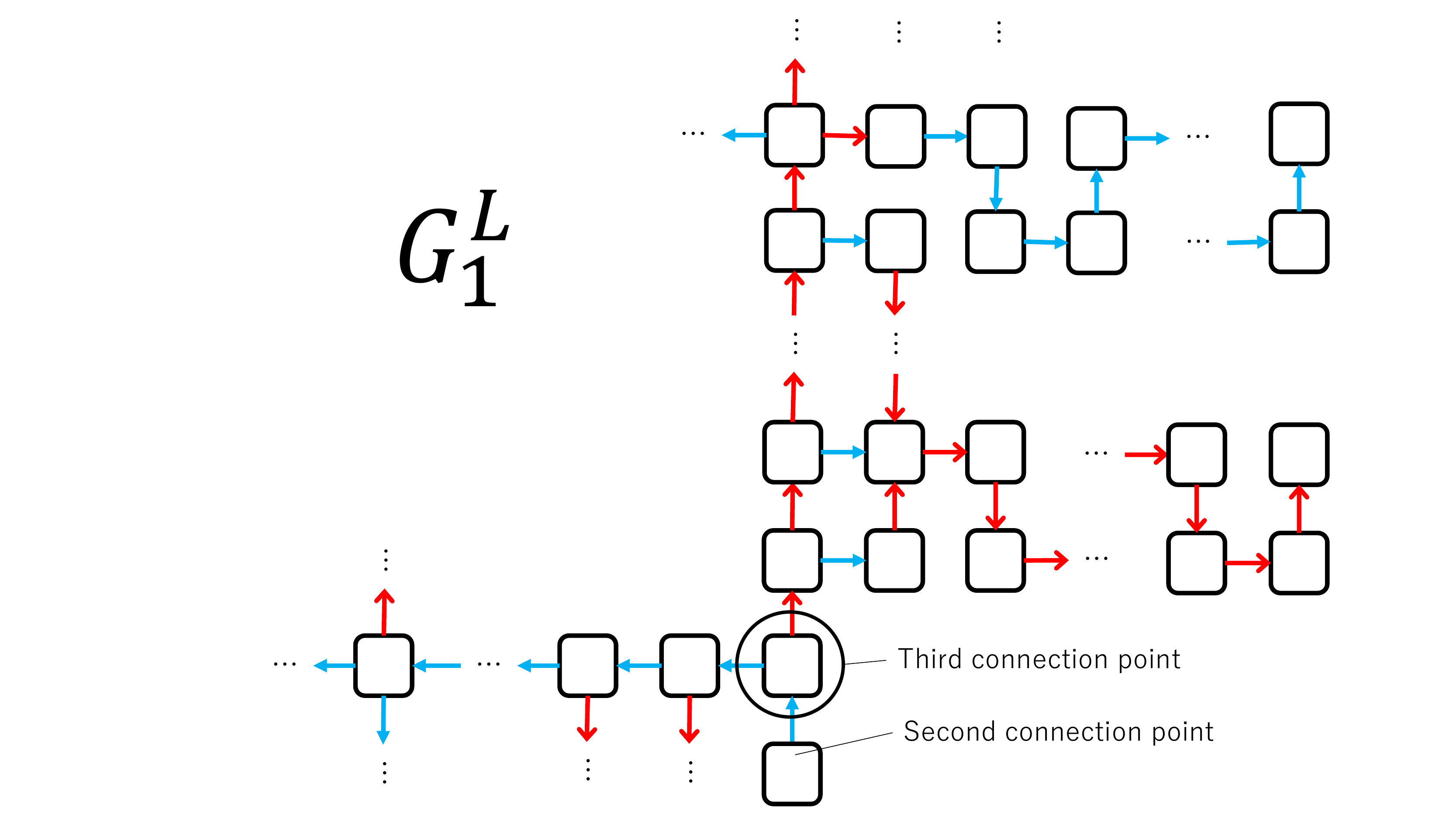}
    \caption{$G^L_1$ connected to third connection point}
    \label{fig:dprove3}
\end{figure}

Let this position be $X.$ We prove that $o(X - G) = \mathcal{P}.$

Suppose that Left moves first. Consider the case Left moves a piece on $X$. Except for the case to move the piece one of the first connection points, Right can win because he can move the piece to the node below.
Then, the value of $X$ changes to $X' = n \cdot \downarrow$ or $X' = n \cdot \downarrow + *$, where $n$ is enough large and satisfies $X' < G$. Then, $o(X' - G) = \mathcal{R}$.

Thus, without loss of generality, we assume that Left moves the piece to the first connection point of $G^L_1$. 

Immediately after the move, Right moves the piece to second connection point of $G^L_1$ and force Left to move third connection point of $G^L_1$ (because otherwise Right can win by moving the piece and get $n' \cdot \downarrow$ or $n' \cdot \downarrow + *$, where $n'$ is enough large), and if Left moves so, then the whole position changes to $G^L_1 - G.$ Therefore, Right can move to $G^L_1 -G^L_1 = 0$ and he wins.
If Right moves first and he moves the piece then Left wins by symmetric discussion.

Consider the case that Right moves first and he changes $X - G$ to $X - G^L_i.$
Without loss of generality, we assume that $i = 1.$  Then Left moves the piece to first connection point of $G^L_1.$ Right has to move the piece to second connection point of $G^L_1$ because otherwise Left can win by moving the piece and get $n''\cdot \uparrow$ or $n'' \cdot \uparrow + *$ for enough large $n''$. However, even if Right moves the piece to second connection point of $G^L_1,$ Left can move the piece to third connection point of $G^L_1$ and then the whole position becomes $G^L_1 - G^L_1 = 0$ and Left wins. If Left moves first and she changes $X-G$ to $X - G^R_i$, then Right wins by symmetric discussion.

Therefore, for $X-G,$ the previous player has a winning strategy. Thus, $X = G.$
\end{proof}

\section{Conclusion}
\label{secCon}
In this paper, we proved {\sc go on lattice} and {\sc beyond the door} are universal partisan rulesets by using game-tree preserving reduction.
The method of reduction has been used primarily for proving computational complexity of problems.
Since this study shows that reduction is also effective in the proof of universality of a game, we can expect that the knowledge accumulated in the study of computational complexity will be utilized in the study of combinatorial game theory.
Further, we also solved an open question on combinatorial game theory by showing that $\delta$-{\sc beyond the door} is a universal partisan dicotic ruleset. 

On the other hand, it is still open questions that whether $\delta$-{\sc blue-red/blue-red-green turning tiles} and $\delta$-{\sc go on lattice} are universal partisan dicotic rulesets or not.

\end{document}